\documentclass[a4paper,12pt]{amsart}

\usepackage{amsthm,amsfonts}
\usepackage{amssymb,graphicx}
\usepackage[all]{xy}

\newtheorem{theorem}{Theorem}[section]
\newtheorem{lemma}[theorem]{Lemma}

\newtheorem{proposition}[theorem]{Proposition}
\newtheorem{definition}[theorem]{Definition}

\newcommand{\bo}{{\bf 0}}
\newcommand{\C}{\mathbb C}

\newcommand{\R}{\mathbb R}
\newcommand{\Z}{\mathbb Z}
\newcommand{\diam}{{\rm diam}}
\newcommand{\supp}{{\rm supp}}
\newcommand{\sphere}{{\bf S}}

\begin{document}

\title[Choking horns in Lipschitz Geometry of Complex Algebraic Varieties.]{Choking horns in Lipschitz Geometry of Complex Algebraic Varieties.}

\author[L. Birbrair]{Lev Birbrair}\thanks{$^a$Research supported under CNPq grant no 300575/2010-6}
\address{Departamento de Matem\'atica, Universidade Federal do Cear\'a
(UFC), Campus do Picici, Bloco 914, Cep. 60455-760. Fortaleza-Ce,
Brasil} \email{birb@ufc.br}

\author[A. Fernandes]{Alexandre Fernandes}\thanks{$^b$Research supported under CNPq grant no 302998/2011-0}
\address{Departamento de Matem\'atica, Universidade Federal do Cear\'a
(UFC), Campus do Picici, Bloco 914, Cep. 60455-760. Fortaleza-Ce,
Brasil} \email{alexandre.fernandes@ufc.br}

\author[V. Grandjean]{Vincent Grandjean}\thanks{$^c$Research supported by CNPq grant no 150555/2011-3}
\address{Departamento de Matem\'atica, Universidade Federal do Cear\'a
(UFC), Campus do Picici, Bloco 914, Cep. 60455-760. Fortaleza-Ce,
Brasil}
\email{vgrandje@fields.utoronto.ca}

\author[D. O'Shea]{Donal O'Shea}\thanks{The fourth author was supported by a faculty fellowship and income
from the Elizabeth T. Kennan chair, both from Mount Holyoke College, together with support from
the mathematics department at UFC in Fortaleza}

\address{President's Office, New College of Florida, 5800 Bayshore Road, Sarasota, FL 34243, USA}
\email{doshea@ncf.edu}

\thanks{The authors wish to thank Alberto Verjovsky and Anne Pichon for valuable conversations and comments.}



\begin{abstract}
We study the Lipschitz Geometry of Complex Algebraic Singularities. For this purpose we introduce the
notion of choking horns. A Choking horn is a family of cycles on the family of the sections of an algebraic
variety by very small spheres centered at a singular point, such that the cycles cannot be boundaries of nearby
chains. The presence of choking horns is an obstruction to metric conicalness as we can see with some classical
isolated hypersurfaces singularities which we prove are not metrically conic. We also show that there exist
infinitely countably many singular varieties, which are locally homeomorphic, but not locally bi-Lipschitz equivalent
with respect to the inner metric.
\end{abstract}

\maketitle

\section{Introduction}
Complex analytic subsets of $\C^n$ comes equipped with two metrics: \em the outer metric \em where the
distance between two points of the subset is their distance measured in $\C^n$ and the \em inner metric
\em where the distance between two points is the infimum of the length of rectifiable curves in the subset
connecting the two points.
It has been known for a long time that complex analytic curves equipped with the inner metric
are, as germs at their singular point, bi-Lipschitz homeomorphic to finitely many cones over a circle.
In other words the local inner geometry of a complex curve at a singular point is that of a metric cone.
This is the geometric analog for curves of the classical result that a complex analytic subset is
topologically conical; that is, in the neighborhood of a singular point, it is homeomorphic to the cone
over its link.

In higher dimensions, however, it has become apparent that despite being topologically conical at singularities,
a complex analytic subset need not be metrically conic.  Indeed, a series of recent papers \cite{BF,BFN1,BFN2,
BFN3,BNP} have shown that the inner geometry of complex analytic subsets is both subtle and rich, and have begun
to work through the structure and details, mostly in the case of isolated singularities of complex surfaces,
of the inner geometry of singularities that are not bi-Lipschitz homeomorphic to a conic subset of some Euclidean space.
In particular the existence of  a \em fast loop \em (a $1$-cycle bounding a $2$-chain that
contracts the cycle faster than linearly \cite{BF,Fe,BNP}) or a
 \em separating set \em (a real hypersurface germ with small density at the singularity which disconnects the germ
of the singularity into two or more connected components each of large density at the singularity \cite{BF,BFN3})
are obstructions to the metric conicalness of the local inner geometry at the singular point.

This paper addresses the following two questions about the local inner geometry of complex analytic singularities:\\
- First, is it possible to identify objects that obstruct the metric conicalness of the local
inner geometry of complex analytic sets near a singular point? \\
- The Brian\c con-Speder family \cite{BFN3} is topologically constant but carries two Lipschitz structures. This motivates our second question:
how many local inner geometries near the singularities can a fixed homeomorphic type of a complex analytic isolated singularity carry? In other words,
can a single topological class have infinitely many Lipschitz models?

In order to answer these questions we introduce a new bi-Lipschitz invariant object, which we call a  \em
choking horn. \em This object is a horn, that is a continuous image of a cylinder over a Euclidean sphere along the
unit interval with the property that one boundary sphere maps onto the singular point and the tangent cone of the
image at the singular point is a real half-line, with the additional property that the intersection of this image
with any Euclidean $(2n-1)$-sphere of sufficiently small radius centered at the singular point is either a non-trivial
cycle in the homology of the link or a trivial cycle which can only be a boundary of a chain with  ``large" diameter.
It turns out that metrically conic singularities cannot admit a choking horn (Theorem \ref{thm:cones-donot-choke}). The choking
horns are higher dimensional analogs of so-called fast loops of the second kind, described in the paper of Birbrair, Neumann and Pichon \cite{BNP}. Moreover, the idea of
choking horn is closely related to Metric Homology \cite{BB1,BB2} and Vanishing Homology \cite{Va}, which study the families of homology cycles (trivial or nontrivial)
with respect to the metric properties near the singular points. Moreover this new object can be realized in some families of Brieskorn hypersurface singularities
as the real locus of the complex singularity. We show that the presence of choking horns in
the family of Brieskorn singularities $(A_{2k+1}^n)_{k\geq 1}$ also implies that any two disjoint elements of this family
cannot have the same local inner geometry at the singular point
(Theorem \ref{thm:not-bilip-homeo}), despite this family being topologically constant
(Theorem \ref{thm:homeo-not-bilip-homeo}).

\medskip
\noindent Notes: 1- Throughout this note, the expression \em bi-Lipschitz homeomorphism \em means \em bi-Lipschitz
and subanalytic homeomorphism. \em This is a harmless liberty since
we never deal with \em bi-Lipschitz homeomorphisms \em that are not subanalytic.
\\
2- Unless explicitly mentioned otherwise, any subset we will work with always comes equipped with the
inner metric.
\section{Elementary properties of choking horns}

Given the germ at the origin of a subset $Z$ of a Euclidean space, the {\it $t$-link $Z_t$ of $Z$} is the locus
of points of $Z$ whose (Euclidean) distance to the origin is equal to $t$.

Let $(X,\bo)$ be the germ of a connected subanalytic subset of an Euclidean space equipped with a metric $d$ which
is either the outer metric $d_{\rm outer}$, or the inner metric $d_{\rm inner}$.

\begin{definition} A \em horn in $X$ \em is the image of a subanalytic continuous map
$$\phi\colon [0,1]\times \sphere^p\rightarrow X,$$
where $\sphere^p$ is the real unit sphere of positive dimension $p$, and such that
\begin{enumerate}
\item[1.] For any $(t,u)\in[0,1]\times \sphere^p$ the image $\phi(t,u)$ lies in $X_t$.
\item[2.] The tangent cone of the image of $\phi$ at $\bo$ is just a single real half-line.
\end{enumerate}
\end{definition}
We observe that a horn is subanalytic and in particular the mapping $\phi$ collapses the boundary
sphere $0\times \sphere^p$ onto the singular point $\bo$.

\begin{definition} Let $Y$ be horn in $(X,d)$. The subset $Y$ is said to \em choke $X$ \em
if for every $t$ small enough and any chain $\xi_t$ with $supp(\xi_t)\subset X_t$ and such that $\partial \xi_t = Y_t$,
then the $d$-diameter of $\xi_t$ satisfies
$$\diam(\supp(\xi_t))\geq Kt$$
for some real number $K>0$ independent of $t$. A horn which chokes $X$ is called a \em choking horn.\em
\end{definition}

The definition of a choking horn in $X$ depends a priori on the choice of the metric $d$ (outer or inner metric) on $X$.
Nevertheless this possible ambiguity is sorted out by the following:
\begin{proposition} Let $Y$ be a horn in $X$. Then, $Y$ is a choking horn with respect to the inner metric if, and only if,
it is a choking horn with respect to the outer metric on $X$.
\end{proposition}
\begin{proof} Whenever $Y$ is a choking horn with respect to the outer metric then it is a choking horn with
respect to the inner metric on $X$. The converse is a direct consequence of Lemma \ref{lem:no-pb-diam} below.
\end{proof}

\begin{lemma}\label{lem:no-pb-diam}
Let $Z$ be a connected subanalytic subset of $X$. Then, there exists a constant $L>0$ such that the inner-diameter of $Z$
is at most $L$ times the outer-diameter of $Z$.
\end{lemma}

\begin{proof} Let $\{\Lambda_i\}_{i=1}^l$ be a subanalytic pancake decomposition of $X$ \cite{Ku,KuO,Mo,Pa}.
Let $x,x'$ be two points on $Z$ such that $d_{\rm inner}(x,x')$ is equal to the inner-diameter of $Z$.
Let $u_0,\dots,u_k$ be a finite sequence of points of $X$ such that $u_0=x$, $u_k=x'$
and $u_j,u_{j-1}$ belong to a same pancake $\Lambda_{i_j}$ of $X$.
Let $\diam(Z)$ be the outer-diameter of $Z$.
By definition of a pancake, we deduce
$$d_{\rm inner}(x,x')\leq \sum_{j=1}^{k} a_j|u_{j}-u_{j-1}|\leq L \diam(Z)$$
where $L = k \max_j\{a_j\}$, for positive real numbers $a_j$, each attached to the corresponding pancake.
\end{proof}

We remark that the lemma above bounding the inner metric from above in terms of the outer metric is what makes the
choice of diameter in measuring the size of bounding cycles in the definition of a choking horn attractive. Other
measures of size, such as volumes, do not seem to admit such easily accessible bounds. The proposition and lemma
above immediately yield the following.

\begin{proposition} Let $(X,\bo)$ be the germ of a connected subanalytic subset of an Euclidean space
equipped with the inner metric.
\begin{enumerate}
\item[1.] The existence of a choking horn in $X$ is invariant by subanalytic bi-Lipschitz homeomorphisms;
\item[2.] If $X$ admits a choking horn $Y$ contained in a sub-germ $W$, then the subset $Y$ is also a choking horn in $W$.
\end{enumerate}
\end{proposition}
%
We recall that a subanalytic germ at the origin of an Euclidean space equipped with the intrinsic metric
is \em metrically conic \em if it is subanalytically bi-Lipschitz homeomorphic to the cone over its link at the origin 
(where the cone over a subset K of Euclidean space is the set of points $\{tx: x \in K, t\in [0,1]\}$).

\smallskip
The first result of this note establishes that choking horns are obstructions to the metric
conicalness of the local inner geometry at the singularity:
\begin{theorem}\label{thm:cones-donot-choke}
A closed subanalytic metrically conic set-germ $(X,0)$ does not admit a choking horn.
%
%
%
%
\end{theorem}

\begin{proof} By hypothesis we can assume that $X$ is the cone over its link with vertex at $\bo$.
Let us consider a subanalytic continuous mapping
$$\phi\colon [0,1]\times \sphere^d\stackrel{\thicksim}\longrightarrow Y\subset X$$
parameterizing a horn $Y$ in $X$. Let $Y_t$ be the $t$-link of $Y$. Since $X$ is a cone
the Hausdorff limit of $\frac{1}{t} Y_t$ as $t\to 0$, taken in the Euclidean unit sphere of the ambient
Euclidean space, is a single point corresponding to a point $x$ in the link of $X$. Hence, for $t>0$ small enough,
there exists a contractible neighborhood $U$ of $x$ in the link of $X$ such that $\frac{1}{t} Y_t\subset U$. Since $U$
is contractible, there is a chain $\xi_t$ with $\supp(\xi_t)\subset U$ satisfying:
$$\diam (\supp(\xi_t))\leq k \diam (\supp(\partial \xi_t)) \ \mbox{and} \ \partial \xi_t =\frac{1}{t}Y_t.$$
Now, consider the chain $\eta_t=t\xi_t$, and push it to the $t$-link of $X$. Note that, $\partial \eta_t = Y_t$ and
 \begin{eqnarray*}
 \diam (\supp(\eta_t)) & = & t\cdot \diam (\supp(\xi_t)) \\
                     & \leq & t\cdot k\cdot \diam (\frac{1}{t}Y_t) \\
                     & = & k\cdot \diam (Y_t) \\
 \end{eqnarray*}
and, since the tangent cone of $Y$ at $\bo$ is just a real half-line, the diameters $\diam (Y_t)$ converge to $0$ (as $t\to 0$) faster
than linearly, hence $\diam (\supp (\eta_t))$ tends to $0$ (as $t\to 0$) faster than linearly as well.
This means exactly that $Y$ does not choke $X$.
\end{proof}
\section{$X_k^n$ Brieskorn hypersurface singularities}

Let $X_k^n$ be the complex isolated hypersurface singularity of $\C^{n+1}$ defined as
$$X_k^n=\{(z_1,\dots,z_n,w)\in\C^{n+1} \ : \ z_1^2+\ldots z^2_n-w^{k}=0\}.$$
We are mostly interested in $X_k^n$ as a germ at the origin $\bo\in\C^{n+1}$.

\smallskip
The next theorem shows
that choking horns exist in the simplest Brieskorn singularities.
%
%
\begin{theorem}\label{thm:Xkn-choked}
For each $n>1$ and each $k>2$, the hypersurface $X^k_n$ admits a choking horn, namely
the real part $X_k^n \cap (\R\times 0)^{n+1}$.
\end{theorem}
\begin{proof}
We are going to prove that the real algebraic hypersurface
$$X_k^n(\R):=\{(x_1,\ldots,x_n,y)\in\R^{n+1} \ : \ x_1^2+\ldots+ x^2_n-y^{k}=0\}$$
embedded as the real part $X_k^n \cap \R^{n+1}$ of $X_k^n$, where $\R^{n+1}\subset \C^{n+1}$
is the usual embedding of $\R^{n+1}$ in $\C^{n+1}$, is a choking horn in $X_k^n$. For this purpose, let us consider the parameterization
$$\phi\colon [0,1]\times \sphere^{n-1}\rightarrow X_k^n$$
of the horn $X_k^n(\R) \subset \C^{n+1}$ defined by $\phi(t,u_1,\ldots,u_n)=((t^{\frac{k}{2}}u_1,0)\ldots,(t^{\frac{k}{2}}u_n,0),(t,0))$
where $\sphere^{n-1}$ is the Euclidean unit sphere centered at origin of $\R^n$.

\smallskip
\noindent{\tt Claim 1.} The $t$-link of $X_k^n(\R)$ is not contractible in $X_k^n\setminus\{w=0\}.$

We start the proof of this claim by showing that for $0 < t\leq 1$ the subset
$$\{(z_1,\dots,z_n)\in \C^n \ : \ z_1^2+\ldots +z_n^2=t\}$$
retracts within $\C^n\setminus\{z_1^2+\ldots +z_n^2=0\}$ to
$$\{(z_1,\dots,z_n)\in \C^n \ : \ z_1,\ldots,z_n \in \R \ \mbox{and} \ z_1^2+\ldots z_n^2\geq t\}.$$
In fact, if we define $z_j:=u_j+iv_j$ with $u_j,v_j\in\R$, then
$$z_1^2+\ldots +z_n^2=t \Leftrightarrow u_1^2+\ldots +u_n^2 - (v_1^2+\ldots +v_n^2) =t \ \mbox{and} \ u_1v_1+\ldots +u_nv_n=0 .$$
On the points $(z_1,\dots,z_n)$ such that $z_1^2+\ldots +z_n^2=t$ we consider the retraction
$$(t,(z_1,\ldots,z_n)) \ \longmapsto (u_1+itv_1,\dots,u_n+itv_n).$$
Then
\begin{eqnarray*}
(u_1+itv_1)^2+\ldots +(u_n+itv_n)^2 &=& u_1^2+\ldots +u_n^2 - t^2(v_1^2+\ldots +v_n^2) -2i(u_1v_1+\ldots +u_nv_n) \\
&=& t + (1-t^2)(v_1^2+\ldots +v_n^2) \\
&\geq t&
\end{eqnarray*}

\noindent
The proof of the claim will be finished once we have showed that the following subset
$$\{(z_1,\dots,z_n)\in \C^n \ : \ z_1,\ldots,z_n \in \R \ \mbox{and} \ z_1^2+\ldots + z_n^2\geq t\}$$
retracts within itself to the subset
$$\{(z_1,\dots,z_n)\in \C^n \ : \ z_1,\ldots,z_n \in \R \ \mbox{and} \ z_1^2+\ldots + z_n^2= t\}.$$
In fact, for each point $(z_1,\dots,z_n)\in \C^n$ such that $z_1,\ldots,z_n \in \R$ and
$z_1^2+\ldots + z_n^2=T\geq t$, we use the $\R_+$-action on $X^n_k$ mapping $(z_1,\ldots,z_n,T)$ to $(s^kz_1,\ldots,s^kz_n,s^2T)$
where $\displaystyle s=\sqrt[2k]{\frac{t}{T}}.$ Thus, we see that the subset
$$X_k^n(t) := \{(z_1,\dots,z_n)\in \C^n \ : \ z_1^2+\ldots +z_n^2=t\}$$
retracts to $X^n_k(\R)$ in $X^n_k\setminus\{w=0\}$.
Let $g\colon\C^n,\bo\rightarrow\C,0$ be the function defined as $g(z_1,\ldots,z_n)=z_1^2+\ldots +z_n^2$.
Since for non-zero $t$ the subset $X_k^n(t)$
%
%
%
%
is the Milnor Fiber of the function $g$, we deduce this subset
is not contractible in $X_k^n(\R)\setminus\{w=0\}.$ Thus Claim 1 is proved.

\medskip
The next argument will end the proof the theorem.
If $\xi_t$ is a chain with support contained in the $t$-link of $X_k^n$ and such that its boundary
$\partial  \xi_t$ is the $t$-link of $X_k^n(\R)$, then
$$\lim_{t\to 0}\frac{1}{t}\diam(\supp(\xi_t))=\frac{\pi}{2}.$$
In particular, for $t>0$ small enough, we get $\diam(\supp(\xi_t))\geq\frac{\pi}{4}t.$
\end{proof}
\begin{theorem}\label{thm:not-bilip-homeo}
For each $n>1$ and any pair $k >l >1$, the complex isolated hypersurface singularities $X_k^n$ and $X_l^n$
are not subanalytically bi-Lipschitz homeomorphic.
\end{theorem}
Before getting into its proof, a most remarkable consequence of Theorem \ref{thm:not-bilip-homeo} is that it answers
the following question: \em can a given local complex singularity topological type carry different
local inner geometries at the singular points? \em
\begin{theorem}\label{thm:homeo-not-bilip-homeo}
For any odd integer $n$, the family of germs $\{(X_{2k+1}^n,\bo)\}_{k>0}$ admits infinitely many bi-Lispchtiz classes, despite
being of constant topological type.
\end{theorem}
\begin{proof}
When $n$ is odd Brieskorn \cite{Br} shows that the link at the origin of $X_{2k+1}^n$ is a $\Z$-homology sphere,
hence is homeomorphic to the Eucldean $(2n-1)$-sphere. From Theorem \ref{thm:not-bilip-homeo}, we deduce
that the given family consists of infinitely many complex algebraic varieties
which are (semialgebraically) homeomorphic but not bi-Lipschitz homeomorphic.
\end{proof}
\begin{proof}[Proof of Theorem \ref{thm:not-bilip-homeo}.] Suppose the contrary. So for some $k>l$,
let $f\colon X_k^n\rightarrow X_l^n$ be a subanalytic bi-Lipschitz homeomorphism. Let us consider $X_k^n(\R)$
the choking horn in $X_k^n$ and the parameterisation
$$\phi\colon [0,1]\times \sphere^{n-1}\rightarrow X_k^n$$
as defined in the proof of the last theorem. So, the image of $f\circ\phi$, which we denote by $Y$,
is a choking horn in $X_l^n$. We claim that the tangent cone of $Y$ at $\bo$ is a real half-line in the
$w$-complex axis (so we may take to be the $y$-real positive half-line after multiplying by $e^{i\theta}$).
Before we prove what we claimed above, let us prove that if $\xi\in\C^{n+1}$ is a vector that does not lie
in the $w$-axis, then for any sufficiently small conical neighborhood of $\xi$, which we denote by $C_{\epsilon}(\xi)$,
there is at least one coordinate hyperplane $H_i=\{z_i=0\}$ such that the projection
$X_l^n\cap C_{\epsilon}(\xi)\rightarrow H_i$ is bi-Lipschitz. In fact, since $\xi\not\in w$-axis at least one of the
first coordinates of $\xi$ must differ from zero. Let us assume, without loss of generality, that $\xi_1\neq 0$
and let us consider the projection on $H_1=\{z_1=0\}$. In this case, locally on $X^n_k$ we have that $z_1$ is
a function depending on $z_2,\dots,z_n,w$ and
$$\frac{\partial z_1}{\partial z_i}=-\frac{z_i}{z_1} \ \mbox{for} \ i=2,\dots,n \ \mbox{and} \
\frac{\partial z_1}{\partial z_i}=\frac{k}{2}\frac{w^{k-1}}{z_1}.$$
Since $z_1$ is bounded away from zero for all $z\in C_{\epsilon}(\xi)$, all partial derivatives above are  bounded.
We conclude that $z_1$ is a locally Lipschitz function, as we desired to show.

Then, if the tangent cone $T_\bo Y$ is tangent to $\xi\not\in w$-axis, once we have proved the projection
$X_l^n\cap C_{\epsilon}(\xi)\rightarrow H_i$ is bi-Lipshitz, we  would have a choking horn in a metrically
conic set which is impossible by the first theorem.

Now, we know that $T_\bo Y$ is the  $y$-real positive half-line. Let us consider the following $\R_+$-action on $X^n_l$:
$$(t,z_1,\dots,z_n,w)\longmapsto t\cdot(z_1,\dots,z_n,w)=(t^{\frac{l}{2}}z_1,\dots,t^{\frac{l}{2}}z_n,tw).$$

Let $u(t),v(t)$ be points belonging to $Y_t$. Then, we know that
$$\| u(t)-v(t) \|\leq d_{X^n_l}(u(t),v(t))\leq \lambda \diam((X^n_k(\R))_t) \approx t^{\frac{k}{2}} \ \mbox{as} \ t\to 0.$$
Choosing the preimages of $u(t)$ and $v(t)$ so their distance represents about the diameter of $t$-link of  $X^n_k$,
we can use the $\R_+$-action on $X^n_l$ to push $Y_t$ to $1$-slice of $X^n_l$. Then,
$$u(t)=(u_1(t),\dots,u_n(t),t) \ \mbox{and} \ v(t)=(v_1(t),\dots,v_n(t),t)$$
whence
$$t^{-1}\cdot u(t)=(t^{-\frac{l}{2}}v_1(t),\dots,t^{-\frac{l}{2}}u_n(t),1) \ \mbox{and} \ t^{-1}\cdot
v(t)=(t^{-\frac{l}{2}}v_1(t),\dots,t^{-\frac{l}{2}}v_n(t),1).$$
Hence
$$\| t^{-1}\cdot u(t)- t^{-1}\cdot v(t) \|\leq \max_{i} t^{-\frac{l}{2}}|u_i(t)-v_i(t)|\lesssim
t^{-\frac{l}{2}} t^{\frac{k}{2}}\to 0 \ \mbox{as} \ t\to 0$$ since $k>l$.

But, then the $1$-slice of $Y$ is homologous to a cycle that gets inside a contractible neighborhood in the
$1$-slice of $X^n_l$, so $Y$ cannot be a choking horn, which is a contradiction. This establishes the proof of the theorem.
\end{proof}
\section{Further comments}
Choking horns give an example of a structure that obstructs metric conicalness in any complex
dimension greater than one. It seems likely that singular germs with a more
sophisticated geometry should contain other objects (neither fast loops, nor separating sets nor choking horns)
which are features of the local inner geometry at the singularity.  A related series of questions comes from the fact
that the choking horns we find above are the real part of a particular realization of the $X^k_n$ singularities.
It is irresistible to ask whether one can detect failure of metric conicalness from the real parts of other realizations.
On a related note, our proofs depended heavily on the $\C^*$ action on the singularity.  Given a real singularity
which is a horn, it is natural to ask whether its complexification must necessarily be metrically conic.  A proof
of anything along these lines would require very different techniques than ours.

\section{Appendix: separating sets. By Walter D. Neumann}

This appendix describes a different proof of some of the results of
this paper. We rely on the paper \cite{BFN3} of Birbrair,
Fernandes and Neumann, in which is shown, among other things:
\begin{theorem}
  If the tangent cone $T_pX$ of a normal complex germ $(X,p)$
  has a complex subcone $V$ of complex codimension $\ge1$ which
  separates $T_pX$, then there is a corresponding separating set in
  $(X,p)$ with tangent cone $V$.
\end{theorem}
The actual theorem in \cite{BFN3} is Theorem 5.1, which deals
with the more general context of real semialgebraic sets, using a
slightly more general definition of ``separating set'' than is needed
in the complex setting. Here we define a \emph{separating set} to be a
real semialgebraic subgerm $(W,p)\subset (X,p)$ whose tangent cone has
real codimension at least $2$ and which separates $X$ into pieces
whose tangent cones have full dimension.

The theorem is applied in \cite{BFN3} to the
example of the Brieskorn variety
$$X=X(a_1,\dots,a_n):=\{(z_1,\dots,z_n)\in \mathbb
C^n~|~z_1^{a_1}+\dots+z_n^{a_n}=0\}\,$$ with $a_1=a_2=a\ge 2$ and
$a_k>a$ for $k>2$.  The tangent cone at the origin is the union of the
a complex hyperplanes $\{z_1 = \xi z_2\}$ with $\xi$ an $a$-th root of
$-1$. These intersect along the $(n - 2)$-plane $V = \{z_1 = z_2 =
0\}$, which separates the tangent cone into $a$ pieces. So $X$ has a
separating set decomposing it into $a$ pieces having tangent cones the
$a$ hyperplanes above. By Brieskorn \cite{Br}, if $n>3$ then the link
$$\Sigma(a_1,a_2,\dots,a_n):=X(a_1,a_2,\dots,a_n)\cap S^{2n-1}$$
is a topological sphere if at least two of the $a_j$'s have no common
factor with any other $a_j$. We thus see examples in any dimension
$n\ge 3$ of singularities with link a topological sphere and having
separating sets which decompose $X$ into arbitrarily many pieces. This
gives a proof of Theorem 3.3
for all dimensions $\ge 3$.

In fact we can be very explicit in this example.  Choose any positive
$\epsilon< 1$. Then putting
$$Y:=\{z\in X~|~ \sum_{i=3}^n |z_i|^{a_i}\le \epsilon(|z_1|^a+|z_2|^a)\},\quad
Z:=\overline{X\setminus Y}\,,$$
the following facts are easily verified:
\begin{enumerate}
\item\label{it1} $T_0Z=\{z\in \mathbb C^n~|~z_1=z_2=0\}$;
\item\label{it2} the image $\pi(Y)$ of the projection of $Y$ to the
  $z_1z_2$-plane has $a$ components;
\item\label{it3} the inverse image of each component of $\pi(Y)$ is a
  component of $Y$, so $Y$ has $a$ components.
\end{enumerate}
Indeed, \eqref{it1} follows immediately from the fact that the
exponents $a_i$ for $i\ge 3$ are greater than $a$. For
\eqref{it2} we note that for any $z=(z_1,\dots,z_n)\in Y$:
\begin{align*}
    |z_1^a+z_2^a|&=|\sum_{i=3}^nz_i^{a_i}|\le \sum_{i=3}^n|z_i|^{a_i}\\&< |z_1|^a+|z_2|^a\,.
\end{align*}
This inequality implies that the coordinates $z_1$ and $z_2$ are both
non-zero, and $z_2/z_1$ cannot be a positive multiple of any $a$-th
root of unity. This condition divides $\pi(Y)$ into $a$ pieces
according to the argument of $z_2/z_1$. For each $a$-th root of $-1$
we denote by $Y'_\xi$ the piece of $\pi(Y)$ which contains points with
$z_2/z_1=\xi$ and denote $Y_\xi=\pi^{-1}(Y'_\xi)$. It is not hard to
check that $Y_\xi$ is connected and its projection to the hyperplane
$z_1=0$ is a bijective map to its image which is bilipschitz in a
neighbourhood of the origin. In fact, $(X,p)=(Y,p)\cup(Z,p)$, glued
along their common boundary, which is topologically the cone over a
disjoint union of $a$ copies of $\Sigma(a_3,\dots,a_n)\times
S^1$.

The previous decomposition $(X,p)=(Y,p)\cup(Z,p)$ is the ``thick-thin
decomposition'' of $(X,p)$. This thick-thin decomposition is discussed in detail for normal surface
singularities in \cite{BNP} (works in progress about such a decomposition in any dimension by  Birbrair, Fernandes, Grandjean,
Neumann, O'Shea, Pichon, Verjovsky). In the case of an isolated complex singularity
germ $(X,p)$ one can construct it as follows: call a tangent line $L$ in $T_pX$ ``very
exceptional'' if no curve in $X$ with tangent $L$ has a metrically
conical neighbourhood. One obtains the thin zone by taking a suitable
horn neighbourhood of the union of all very exceptional tangent lines,
and the thick zone is then the closure of the complement of the thin zone.

Similarly, the example $X(2,\dots,2,k)\subset\mathbb C^{n+1}$ of the
body of this paper has thick-thin decomposition whose thin part is a
$k/2$-horn neighbourhood of the $z_{n+1}$-axis, which has boundary the
cone over $S^1\times \Sigma(2,\dots,2)$.

\end{document}